\documentclass[12pt]{amsproc}

\theoremstyle{definition}

\theoremstyle{remark}

\numberwithin{equation}{section}
\usepackage{graphicx}
\usepackage{latexsym}
\usepackage{amsmath}
\usepackage{amssymb}
\usepackage{amsfonts}
\usepackage{verbatim}
\usepackage{mathrsfs}
\usepackage[latin1]{inputenc}
\usepackage{color}
\usepackage[colorlinks,citecolor=blue,urlcolor=blue]{hyperref}

\newtheorem{tm}{Theorem}[section]
\newtheorem{rk}{Remark}[section]
\newtheorem{ap}{Assumption}[section]

\newtheorem{lm}{Lemma}[section]

\newcommand{\cc}{\mathbb C}
\newcommand{\ee}{\mathbb E}
\newcommand{\pp}{\mathbb P}
\newcommand{\nn}{\mathbb N}
\newcommand{\rr}{\mathbb R}
\newcommand{\hh}{\mathbb H}

\newcommand{\bs}{\mathbf s}

\newcommand{\CC}{\mathcal C}
\newcommand{\LL}{\mathcal L}

\newcommand{\OOO}{\mathscr O}
\newcommand{\FFF}{\mathscr F}
\newcommand{\MMM}{\mathscr M}
\newcommand{\HHH}{\mathscr H}

\allowdisplaybreaks \allowdisplaybreaks[4]

\begin{document}

\title[Optimal Regularity of SEEs in UMD Banach Spaces]
{Optimal Regularity of Stochastic Evolution Equations in M-type 2 Banach Spaces}

\author{Jialin Hong}
\address{LSEC, ICMSEC, Academy of Mathematics and Systems Science, Chinese Academy of Sciences, Beijing 100190, China; School of Mathematical Sciences, University of Chinese Academy of Sciences, Beijing 100049, China}
\curraddr{}
\email{hjl@lsec.cc.ac.cn}
\thanks{}
\author{Chuying Huang}
\address{LSEC, ICMSEC, Academy of Mathematics and Systems Science, Chinese Academy of Sciences, Beijing 100190, China; School of Mathematical Sciences, University of Chinese Academy of Sciences, Beijing 100049, China}
\curraddr{}
\email{huangchuying@lsec.cc.ac.cn}

\author{Zhihui Liu}
\address{Department of Applied Mathematics, 
The Hong Kong Polytechnic University, Hongkong}
\curraddr{}
\email{liuzhihui@lsec.cc.ac.cn (Corresponding author)}
\thanks{}



\date{\today}

\dedicatory{}

\begin{abstract}
In this paper, we prove the well-posedness and optimal trajectory regularity for the solution of stochastic evolution equations driven by general multiplicative noises in martingale type 2 Banach spaces.
The main idea of our method is to combine the approach in \cite{HL17} dealing with Hilbert setting and a version of Burkholder inequality in M-type 2 Banach space.
Applying our main results to the stochastic heat equation gives a positive answer to an open problem proposed in \cite{JR12(JDE)}.
\end{abstract}

\maketitle


\section{Introduction}
\label{sec1}

In this paper, we investigate the well-posedness and optimal trajectory regularity for the solution of the stochastic evolution equation
\begin{align}\label{spde} \tag{SEE}
\begin{split}
&{\rm d}X(t)=(AX(t)+F(X(t))) {\rm d}t+G(X(t)) {\rm d}W_U(t),
\ t\in (0,T];\\
& X(0)=X_0 
\end{split}
\end{align}
in a martingale type (M-type) 2 Banach space $(E,\|\cdot\|)$, under weak assumptions on its data.
Here $T$ is a fixed positive number, $A$ is a generator of an analytic $C_0$-semigroup $S(\cdot)$ on $E$ and $W_U:=\{W_U(t):\ t\in [0,T]\}$ is a cylindrical Wiener process in a separable Hilbert space $U$ with respect to a stochastic basis $(\Omega,\FFF,(\FFF_t)_{t\in [0,T]},\pp)$, i.e., for any $g,h\in U$, $W_Uh=\{(W_U(t)h:\ t\in [0,T] \}$ is a Brownian motion, and 
$\ee[W_U(s)g\cdot W_U(t)h]=(s\wedge t)(g,h)_U$ for all $s,t\in [0,T]$.

The well-posedness and regularity for the solution of Eq. \eqref{spde} in Hilbert setting have been studied extensively; see, e.g., 
G. Da prato, S. Kwapie\u{n} and J. Zabczyk \cite{DKZ87(STO)}, 
A. Jentzen and M. R\"ockner \cite{JR12(JDE)},
J. Hong and Z. Liu \cite{HL17} and references therein.
There are several reasons to consider Eq. \eqref{spde} in Banach setting instead of Hilbert setting. 
On one hand, it is shown that a large class of equations with various applications are more suitable to be described in the M-type 2 Banach space $E=\mathbb L^q(\OOO)$ with $q\ge 2$, where throughout $\OOO$ is a bounded, open subset of $\rr^{\rm d}$ with regular boundary (see, e.g., \cite{Brz95(PA), Brz97(SSR)}  for Wiener case and \cite{DGT12(RTRF)} for rough case). 
On the other hand, some error estimations in numerical analysis and simulation of Eq. \eqref{spde} can also be improved (see, e.g., \cite{CHJNW17(IMA), HJK16}).

Another motivation for us to study the optimal trajectory regularity for the solution of Eq. \eqref{spde} in M-type 2 Banach space is an open problem proposed in \cite[Section 4]{JR12(JDE)}.
There the authors gave certain regularity for the solution of Eq. \eqref{spde} in Hilbert setting and applied their result to the stochastic heat equation in $\hh=\mathbb L^2((0,1)^{\rm d})$ driven by multiplicative ${\bf Q}$-Wiener process with certain Lipschitz-type nonlinear drift and diffusion coefficients (see Eq. \eqref{she}). 
Their main result showed that the solution of this equation enjoys $\hh^\sigma$-regularity for any $\sigma\in [0,\min\{3/2,\epsilon+1\})$, where $\hh^\sigma$ denotes the domain of $(-\Delta)^\frac\sigma2$ and $\epsilon\in (0,1]$ is the same H\"older constant for the eigenfunctions of ${\bf Q}$. 
Whether this type of spatial regularity holds, for the solution of Eq. \eqref{she} in the M-type 2 Banach space $E=\mathbb L^q((0,1)^{\rm d})$ for $q>2$, was proposed by the authors of \cite{JR12(JDE)} as an open problem.
If this is valid for sufficiently large $q$, then by using appropriate
Sobolev embeddings one can even show that the solution possesses values in the spatio-temporal H\"older space 
$\CC^\delta([0,T]; \CC^\kappa((0,1)^{\rm d}))$ for certain 
$\delta,\kappa\ge 0$.

In \cite{HL17}, the first and third authors of the present paper established the optimal trajectory regularity for the solution of Eq. \eqref{spde} in Hilbert setting under less assumptions on its data.
The main aim in this paper is to generalize the main results, Theorems 2.1--2.4, in \cite{HL17} under Hilbert setting to Banach setting, by combining the approach in \cite{HL17} and a version of Burkholder inequality in M-type 2 Banach space given by Brze\'zniak \cite{Brz97(SSR)} (see \eqref{bdg}).
Applying our main results, Theorems \ref{main1} and \ref{main2}, to the stochastic heat equation \eqref{she} gives a positive answer to the aforementioned open problem in \cite{JR12(JDE)}.

We also note that a theory of stochastic integration and stochastic evolution equation in UMD Banach spaces  (i.e., spaces in which martingale differences are unconditional) had been developed by J. van Neerven, M. Veraar and L. Weis \cite{NVW07(AOP), NVW08(JFA)}.
Various classical spaces, such as $\mathbb L^q(\OOO)$ for $q\in (1,2)$, do have the UMD property but fail to have M-type 2.
On the other hand, an M-type 2 Banach space needs not to be of  UMD; see an example given by J. Bourgain \cite{Bou83(AM)}.
Meanwhile, in the application to the derivation of spatial H\"older regularity for the solution of the stochastic heat equation, it is sufficient to consider the M-type 2 Banach space $\mathbb L^q(\OOO)$ for sufficiently large $q$.

The rest of the paper is organized as follows.
In the next section, we give preliminaries of the stochastic integration in M-type 2 Banach setting and derive a version of Burkholder inequality.
The well-posedness and optimal trajectory regularity for the solution of Eq. \eqref{spde} are then established in Sections \ref{sec3} and \ref{sec4}, respectively.
Finally, we illustrate our regularity result by the stochastic heat equation and settle the aforementioned problem in the last section.

\section{Burkholder Inequality in M-type 2 Banach Spaces}
\label{sec2}

In this section, we give some preliminaries of the stochastic integration in M-type 2 Banach space and a version of Burkholder type inequality. 
For more details about definitions and properties of M-type 2 Banach space, we refer to \cite{Brz95(PA), Brz97(SSR)}.

It is known that the stochastic calculation in Banach space depends heavily on the geometric structure of the underlying space.
We first recall the definitions of type and M-type for a Banach space.
Let $\{\epsilon_n\}_{n\in \nn_+}$ be a Rademacher sequence in a probability space $(\Omega',\FFF',\pp')$, i.e., a sequence of independent random variables taking the values $\pm 1$ with probability $1/2$.
Then $E$ is called of type $2$ if there exists a constant $\tau\ge 1$ such that 
\begin{align*}
\bigg\|\sum_{n=1}^N \epsilon_n x_n\bigg\|_{\mathbb L^2(\Omega';E)}
\le \tau \bigg(\sum_{n=1}^N \|x_n\|^2 \bigg)^\frac12
\end{align*}
for all finite sequences $\{x_n\}_{n=0}^N$ in $E$;
$E$ is called of M-type $2$ if there exists a constant $\tau^M \ge 1$ such that  
\begin{align}\label{m-type}
\|f_N\|_{\mathbb L^2(\Omega;E)}
\le \tau^M \bigg(\|f_0\|_{\mathbb L^2(\Omega';E)}^2
+\sum_{n=1}^N \|f_n-f_{n-1}\|_{\mathbb L^2(\Omega';E)}^2 \bigg)^\frac12
\end{align}
for all $E$-valued $\mathbb L^p$-martingales $\{f_n\}_{n=0}^N$.
It is well-known that martingale type $2$ implies type $2$ and vice verse for UMD spaces.
It is clear that when $E$ is an M-type 2 Banach space then so is $E^\theta$ for $\theta>0$.
In the rest of the paper, we assume that $E$ is an M-type 2 Banach space if without further emphasis.

For stochastic integration in Banach space, $\gamma$-radonifying operators play an important role instead of Hilbert--Schmidt operators in Hilbert setting. 
Let $\{\gamma_n\}_{n\geq 1}$ is a sequence of independent 
$\mathcal N(0,1)$-random variables on a probability space $(\Omega',\FFF',\pp')$.
An operator $R\in \LL(U,E)$ is called $\gamma$-radonifying if there exists an orthonormal basis $\{h_n\}_{n\in \nn_+}$ of $U$ such that the Gaussian series $\sum_{n\in \nn_+}\gamma_nRh_n$ converges in $\mathbb L^2(\Omega';E)$.
Then the number 
\begin{align*}
\|R\|_{\gamma(U,E)}:=\Big\|\sum_{n\in \nn_+}\gamma_nRh_n\Big\|_{\mathbb L^2(\Omega';E)}
\end{align*}
does not depend on the sequence $\{\gamma_n\}_{n\geq 1}$ and the basis $\{h_n\}_{n\in \nn_+}$, and
it defines a norm on the space $\gamma(U,E)$ of all $\gamma$-radonifying operators from $U$ into $E$.
If $E$ is a Hilbert space, then $\gamma(U,E)=\LL_2(U,E)$ isometrically, where $\LL_2(U,E)$ denotes the space of
all Hilbert--Schmidt operators from $U$ to $E$.
Moreover, the $\gamma$-radonifying operator satisfies the ideal property, i.e., for any $S_1\in \LL(U',U)$ and $S_2\in \LL(E,E')$ with Hilbert space $U'$ and Banach space $E'$ it holds that 
\begin{align}\label{rad-ide}
||S_2RS_1||_{\gamma(U',E')}\le \|S_2\|_{\LL(E,E')}||R||_{\gamma(U,E)}\|S_1\|_{\LL(U',U)}.
\end{align}

An $\LL(U,E)$-valued adapted process $\Phi$ on $(\Omega, \FFF, \FFF_t,\pp)$ is said to be elementary, if there exists a partition $\{t_n\}_{n=0}^N$ of $[0,T]$, a sequence of disjoint sets $\{A_{mn}\}_{m=1}^M\subset \FFF_{t_n}$ for each $0\le n\le N$, orthonormal elements $\{h_k\}_{k=1}^K \subset U$ and $\{x_{kmn}\}_{k=1,m=1,n=0}^{K,M,N} \subset E$ such that
\begin{align*}
\Phi(t,\omega)
=\sum_{n=0}^{N-1} \sum_{m=1}^M 
\chi_{(t_n, t_{n+1}]\times A_{mn}}(t,\omega)
\sum_{k=1}^K h_k\otimes x_{kmn}, 
\ (t,\omega)\in [0,T]\times \Omega.
\end{align*}
Then one can define the stochastic integral of an elementary process $\Phi$ with respect to the $U$-cylindrical Wiener process $W_U$ by
\begin{align*}
\int^t_0\Phi(r){\rm d} W_U(r)
:=\sum_{n=0}^{N-1} \sum_{m=1}^M \sum_{k=1}^K 
\chi_{A_{mn}} [W_U(t_{n+1}\wedge t)-W_U(t_n\wedge t)] h_k x_{kmn}
\end{align*}
for $t\in [0,T]$.
It was shown in \cite[Theorem 2.4]{Brz97(SSR)} that 
\begin{align*}
\bigg\|\int_0^t \Phi(r){\rm d} W_U(r)\bigg\|_{\mathbb L^2(\Omega;E)}
\le \tau^M \|\Phi\|_{\mathbb L^2(\Omega; \mathbb L^2(0,t; \gamma(U,E)))},
\quad t\in [0,T],
\end{align*}
where $\tau^M$ is the M-type 2 constant of $E$ in \eqref{m-type}.

It is a routine density argument to extend the stochastic integral
to arbitrary $\LL(U,E)$-valued predictable processes $\Phi$ such that $\|\Phi\|_{\mathbb L^2(\Omega; \mathbb L^2(0,T; \gamma(U,E)))}<\infty$; the process 
$\int_0^\cdot \Phi(r){\rm d} W_U(r)$ is a continuous martingale.
Then the Doob's martingale inequality leads to the one-sided Burkholder inequality:
\begin{align*}
\ee \bigg[\sup_{t\in [0,T]} \bigg\|\int_{0}^t \Phi(r){\rm d}W_U(r)\bigg\|^p \bigg] 
\le C \|\Phi\|^p_{\mathbb L^p(\Omega;\mathbb L^2(0,T;\gamma(U,E)))}
\end{align*}
for any $p\ge 2$.
Applying Minkovskii inequality, we obtain the following version of one-sided Burkholder inequality:
\begin{align}\label{bdg}
\ee \bigg[\sup_{t\in [0,T]} \bigg\|\int_{0}^t \Phi(r){\rm d}W_U(r)\bigg\|^p \bigg] 
\le C \|\Phi\|^p_{\mathbb L^2(0,T; \mathbb L^p(\Omega;\gamma(U,E)))}.
\end{align}

\section{Well-posedness}
\label{sec3}

Assume that the linear operator $A: D(A)\subseteq E\rightarrow E$ is the infinitesimal generator of an analytic $C_0$-semigroup $S(\cdot)$ and the resolvent set of $A$ contains all $\lambda\in \cc$ with $\Re [\lambda]\ge 0$.
Then one can define the fractional powers $(-A)^\theta$ for $\theta\in \rr$ of the operator $-A$.
Let $\theta\ge 0$ and $E^\theta$ be the domain of $(-A)^\frac\theta2$ equipped with the norm $\|\cdot\|_\theta$:
\begin{align*}
\|x\|_\theta:=\|(-A)^\frac\theta2 x\|,\quad x\in E^\theta.
\end{align*}
Note that $E^{\theta}$ is also an M-type $2$ Banach space.
The analyticity of $S(\cdot)$ ensures the following properties (see, e.g., \cite[Theorem 6.13 in Chapter 2]{Paz83}):
\begin{align}\label{ana}
&\|(-A)^\nu \|_{\LL(E)}\le C, \nonumber\\
&\|(-A)^\mu S(t)\|_{\LL(E)}\le C t^{-\mu},  \\
&\|(-A)^{-\rho} (S(t)-{\rm Id}_{E}) \|_{\LL(E)}\le Ct^\rho, \nonumber
\end{align}
for any $0<t\le T$, $\nu \le 0\le \mu$ and $0\le \rho\le 1$.

Recall that a predictable stochastic process $X:[0,T]\times \Omega\rightarrow \hh$  is called a mild solution of Eq. \eqref{spde} if $X\in \mathbb L^\infty(0,T;E)$ a.s
and for all $t\in [0,T]$ it holds a.s. that 
\begin{align}\label{mild}
X(t)=S(t)X_0+S*F(X)(t)+S\diamond G(X)(t),
\end{align}
where $S*F(X)$ and $S\diamond G(X)$ denote the deterministic and stochastic convolutions, respectively:
\begin{align*}
&S*F(X)(\cdot):=\int_0^\cdot S(\cdot-r) F(X(r)){\rm d}r, \\
&S\diamond G(X)(\cdot):=\int_0^\cdot S(\cdot-r) G(X(r)){\rm d}W(r).
\end{align*}
The uniqueness of the mild solution of Eq. \eqref{spde} is understood in the sense of stochastically equivalence.

To perform the well-posedness result, we give the following Lipschitz-type continuity with linear growth condition on the nonlinear operators $F$ and $G$.

\begin{ap} \label{a1}
There exist four nonnegative, Borel measurable functions $K_F,K_{F,\theta}$ and $K_G,K_{G,\theta}$ on $[0,T]$ with
\begin{align*}
K_F^0:=\int_0^T K_F(t) {\rm d}t<\infty, \quad 
&K_G^0:=\int_0^T K_G^2(t) {\rm d}t<\infty, \\
K_F^\theta:=\int_0^T K_{F,\theta} (t) {\rm d}t<\infty, \quad 
&K_G^\theta:=\int_0^T K_{G,\theta}^2(t) {\rm d}t<\infty,
\end{align*}
such that for any $x,y\in E$ and $z\in E^\theta$ it holds that
\begin{align*}
\|S&(t) (F(x)-F(y))\| \le K_F(t)\|x-y\|, \\
&\|S(t) F(z)\|_\theta \le K_{F,\theta} (t)(1+\|z\|_\theta), 
\end{align*}
and that 
\begin{align*}
\|S&(t) (G(x)-G(y))\|_{\gamma(U,E)} \le K_G(t)\|x-y\|, \\
&\|S(t) G(z)\|_{\gamma(U,E^\theta)} \le K_{G,\theta}(t)(1+\|z\|_\theta). 
\end{align*}
\end{ap}

\begin{tm} \label{main1}
Let $p\ge 2$, $\beta\geq\theta\geq 0$, $X_0:\Omega\rightarrow E^\beta$ be strongly $\FFF_0$-measurable such that $X_0\in \mathbb L^p(\Omega;E^\beta)$ and Assumptions \ref{a1} holds.
Then Eq. \eqref{spde} exists a unique mild solution $X$ such that the following statements hold.
	\begin{enumerate}
		\item
		There exists a constant $C=C(T,p,K^\theta_{F},K^\theta_{G})$ such that
		\begin{align}\label{mom}
		\sup_{t\in [0,T]}\ee\Big[\|X(t)\|_{\theta}^p\Big]
		&\le C\Big(1+\ee\Big[\|X_0\|_{\theta}^p \Big]\Big).
		\end{align}
		
		\item
		The solution $X$ is continuous with respect to $\|\cdot\|_{\mathbb L^p(\Omega;E^{\theta)}}$:
		\begin{align}\label{mean}
		\lim_{t_1\rightarrow t_2}\ee\Big[\|X(t_1)-X(t_2)\|_{\theta}^p\Big]=0,
		\quad t_1,t_2\in [0,T].
		\end{align}
	\end{enumerate}
\end{tm}

\begin{proof}
		
	For $p\ge 2$, we denote by
	$\HHH_{\theta}^p$ the space of all $E^{\theta}$-valued processes $Y$ defined on $[0,T]$ such that 
	\begin{align*}
	\|Y\|_{\HHH_{\theta}^p}
	:=\sup_{t\in [0,T]} \Big(\ee\Big[\|Y(t)\|_{\theta}^p \Big]\Big)^\frac1p<\infty.
	\end{align*}
	Note that $(\HHH_{\theta}^p,\|\cdot\|_{\HHH_{\theta}^p})$ becomes a Banach space after identifying stochastic processes which are stochastically equivalent.
	
For $X_0\in \mathbb L^p(\Omega; E^\beta)\subseteq \mathbb L^p(\Omega; E^{\theta})$ and $X\in \HHH_{\theta}^p$, define an operator $\MMM$ by
\begin{align*}
\MMM(X)(t)=S(t)X_0+S*F(X)(t)+S\diamond G(X)(t),
\quad t\in [0,T].
\end{align*}
We first show that $\MMM$ maps from $\HHH_{\theta}^p$ to 
	$\HHH_{\theta}^p$. 
	
	By Minkovskii inequality, we get
	\begin{align*}
	\big\|\MMM(X)\big\|_{\HHH_{\theta}^p}
	&\le \big\|S(t)X_0\big\|_{\HHH_{\theta}^p}
	+\big\|S*F(X)(t) \big\|_{\HHH_{\theta}^p}
	+\big\|S\diamond G(X)(t)\big\|_{\HHH_{\theta}^p}.
	\end{align*}
Since $S$ is uniformly bounded in $E$, we set 
	\begin{align*}
	M(t):=\sup_{r\in [0,t]} \|S(r)\|_{\LL(E)},\quad t\in [0,T].
	\end{align*}
	Then 
	\begin{align*}
	\big\|S(t)X_0\big\|_{\HHH_{\theta}^p}
	\le M_T  \|X_0\|_{\mathbb L^p(\Omega;E^{\theta})}.
	\end{align*}
	By Minkovskii inequality and Assumption \ref{a2}, we get
	\begin{align*}
	\| S*F(X)(t) \|_{\HHH_{\theta}^p} 
	&\le \sup_{t\in [0,T]} \int_0^t \|S(t-r) F(X(r))\|_{\mathbb L^p(\Omega;E^{\theta})} {\rm d}r \\
	&\le \sup_{t\in [0,T]} \int_0^t K_{F,\theta}(t-r) (1+\|X(r)\|_{\mathbb L^p(\Omega;E^{\theta})}) {\rm d}r \\
	&\le \bigg(\int_0^T K_{F,\theta}(r){\rm d} r\bigg) \bigg(1+\|X\|_{\HHH_{\theta}^p}\bigg).
	\end{align*}
	For the stochastic convolution, applying the Burkholder inequality \eqref{bdg} and Assumption \ref{a2}, we obtain
	\begin{align*}
	\ee\bigg[\|S\diamond G(X)(t)\|^p_\theta\bigg] 
	&\le \bigg(\int_0^t  \|S(t-r) G(X(r)) \|_{\mathbb L^p(\Omega; \gamma(U,E^\theta))}^2  {\rm d} r\bigg)^\frac p2 \\
	&\le \bigg(\int_0^t K_{G,\theta}^2(t-r) 
	\big(1+\|X(r)\|_{\mathbb L^p(\Omega; E^\theta)}  \big)^2
	{\rm d}r\bigg)^\frac p2 \\
	&\le \bigg(\int_0^t K_{G,\theta}^2(r)  {\rm d}r\bigg)^\frac p2
	\bigg(1+\|X\|_{\HHH_{\theta}^p} \bigg)^p.
	\end{align*}
	Then 
	\begin{align*}
	\|S\diamond G(X)(t)\|_{\HHH_{\theta}^p} 
	\le \bigg(\int_0^T K_{G,\theta}^2(r)  {\rm d}r\bigg)^\frac12
	\bigg(1+\|X\|_{\HHH_{\theta}^p} \bigg).
	\end{align*}
	Combining the above estimations, we get
	\begin{align*}
	\big\|\MMM(X)\big\|_{\HHH_{\theta}^p}
	&\le M_T  \|X_0\|_{\mathbb L^p(\Omega;E^\theta)}
	+N_T \big(1+\|X\|_{\HHH_{\theta}^p} \big),
	\end{align*}
	where $N(t)$ is the non-decreasing, continuous function defined by
	\begin{align*}
	N(t)=
	\int_0^t K_{F,\theta}(r){\rm d} r+\bigg(\int_0^t K_{G,\theta}^2(r)  {\rm d}r\bigg)^\frac12,\quad t\in [0,T].
	\end{align*}
	Thus $\big\|\MMM(X)\big\|_{\HHH_{\theta}^p}<\infty$ and $\MMM$ maps from $\HHH_{\theta}^p$ to $\HHH_{\theta}^p$. 
	
Next we show that $\MMM$ is a contraction in $\HHH^p_0$. 
	To this end, we introduce the norm
	\begin{align*}
	\|Y\|_{\HHH^{p,u}}
	:=\sup_{t\in [0,T]} e^{-ut}\bigg(\ee\bigg[\|Y(t)\|^p \bigg]\bigg)^\frac1p<\infty,
	\end{align*}
	which is equivalent to $\|\cdot\|_{\HHH^p_0}$ for any $u>0$.
	Then for $X,Y\in \HHH^p_0$, previous arguments yield that
	\begin{align*}
	& \|\MMM(X)(t)-\MMM(Y)(t)\|_{\mathbb L^p(\Omega;E)} \\
	&\le \int_0^t K_F(t-r) \|X(r)-Y(r)\|_{\mathbb L^p(\Omega;E)} {\rm d}r \\
	&\quad +\bigg(\int_0^t K_G^2(t-r) \|X(r)-Y(r)\|_{\mathbb L^p(\Omega;E)}^2 {\rm d}r \bigg)^\frac12 \\
	&\le \bigg(\int_0^t e^{u r} K_F(t-r){\rm d}r 
	+\bigg(\int_0^t e^{2u r} K_G^2(t-r) {\rm d}r\bigg)^\frac12 \bigg) 
	\|X-Y\|_{\HHH^{p,u}}.
	\end{align*}
As a result,
	\begin{align*}
	\|\MMM(X)(t)-\MMM(Y)(t)\|_{\HHH^{p,u}}
	\le L_T(u) \|X-Y\|_{\HHH^{p,u}},
	\end{align*}
	where 
	\begin{align*}
	L_T(u)=\int_0^T e^{-ur} K_F(r){\rm d}r 
	+\bigg(\int_0^T e^{-2u r} K_G^2(r) {\rm d}r\bigg)^\frac12.
	\end{align*}
It is not difficult to show that there exists a sufficiently large $u^*>0$ such that $L_T(u^*)<1$.
Thus the operator $\MMM$ is a strict contraction in 
	$(\HHH_\theta^p,\|\cdot\|_{\HHH^{p,u^*}})$, which shows the existence and uniqueness of a mild solution of Eq. \eqref{spde} such that
	\begin{align}\label{mom-0}
	\sup_{t\in [0,T]} \ee\bigg[\|X(t)\|_\theta^p\bigg]<\infty.
	\end{align}
	
	Now we prove the moment estimation \eqref{mom}.
	Previous procedure implies the following estimation:
	\begin{align*}
	\|X(t)\|_{\mathbb L^p(\Omega;E^\theta)} 
	&\le M(t)  \|X_0\|_{\mathbb L^p(\Omega;E^\theta)}+N(t)
	+\int_0^t K_{F,\theta}(t-r) \|X(r)\|_{\mathbb L^p(\Omega;E^\theta)} {\rm d}r  \\
	&\quad +\bigg(\int_0^t K_{G,\theta}^2(t-r) \|X(r)\|_{\mathbb L^p(\Omega; E^\theta)}^2 {\rm d}r\bigg)^\frac12.
	\end{align*}
	Then by H\"older inequality, we have
	\begin{align*}
	\|X(t)\|_{\mathbb L^p(\Omega;E^\theta)}^2 
	&\le m(t)+\int_0^t K(t-r) \|X(r)\|_{\mathbb L^p(\Omega;E^\theta)}^2 {\rm d}r,
	\end{align*}
where $m(\cdot):=3 (M(\cdot)  \|X_0\|_{\mathbb L^p(\Omega;E^\theta)}+N(\cdot) )^2$ is non-decreasing and 
$K(\cdot):=3 K_F^\theta K_{F,\theta}(\cdot)+3K_{G,\theta}^2(\cdot)$ is integrable on $[0,T]$.
Applying the Gronwall's inequality in \cite[Lemma 3.1]{HL17} and the uniform boundedness \eqref{mom-0}, we get \eqref{mom}.
	
	It remains to prove the $\mathbb L^p(\Omega)$ continuity.
	Without loss of generality, assume that $0\le t_1<t_2\le T$.
	Due to the strong continuity of the 
	$C_0$-semigroup $S(t)$: 
	\begin{align}
	\label{c0}
	(S(t)-{\rm Id}_{E}) x\rightarrow 0\ \text{in}\ E
	\ \text{as}\ t\rightarrow 0, \quad \forall\ x\in E,
	\end{align} 
	the term $S(\cdot)X_0$ is continuous $\mathbb L^p(\Omega;E^\theta)$:
	\begin{align}\label{con-mean1}
	& \lim_{t_1\rightarrow t_2}
	\ee\bigg[\big\|S(t_1) X_0-S(t_2) X_0 \big\|^p_\theta\bigg] \nonumber \\
	&=\lim_{t_1\rightarrow t_2}\ee\bigg[\big\|(S(t_2-t_1)-{\rm {\rm Id}_{E}}) S(t_1)X_0 \big\|^p_\theta\bigg]
	=0.
	\end{align}
	
	Next we consider the stochastic convolution $S\diamond G(X)$.
	By (\ref{bdg}), we get
	\begin{align*}
	&\ee\bigg[\big\|S\diamond G(X)(t_1)-S\diamond G(X)(t_2) \big\|_\theta^p\bigg] \\
	&\le \bigg( \int_0^{t_1} \|(S(t_2-t_1)-{\rm Id}_E) S(t_1-r) G(X(r))\|_{\mathbb L^p(\Omega;\gamma(U,E^\theta))}^2 {\rm d} r \bigg)^\frac p2  \\
	&\quad +\bigg(\int_{t_1}^{t_2} \|S(t_2-r) G(X(r))\|_{\mathbb L^p(\Omega;\gamma(U,E^\theta))}^2 {\rm d} r \bigg)^\frac p2
	=: I_1+I_2.
	\end{align*}
	For the first term, by the uniformly boundedness \eqref{mom} of $X$, we get 
	\begin{align*}
	I_1\le C \bigg(\int_0^{t_1} K_{G,\theta}^2(r)  {\rm d}r\bigg)^\frac p2
	\bigg(1+\|X\|_{\HHH_{\theta}^p} \bigg)^p
	<\infty.
	\end{align*}
	Then $I_1$ tends to 0 as $t_1\rightarrow t_2$ by the strong continuity \eqref{c0} of the $C_0$-semigroup $S(\cdot)$ and Lebesgue dominated convergence theorem.
	For the second term, we have
	\begin{align*}
	I_2\le \bigg(\int_0^{t_2-t_1} K_{G,\theta}^2 (r) {\rm d}r\bigg)^\frac p2
	\bigg(1+\|X\|_{\HHH^p_\theta} \bigg)^p
	\rightarrow 0\quad \text{as}\quad  t_1\rightarrow t_2
	\end{align*}
	by Lebesgue dominated convergence theorem.
	Therefore,
	\begin{align}\label{con-mean2}
	\lim_{t_1\rightarrow t_2}
	\ee\bigg[\big\|S\diamond G(X)(t_1)-S\diamond G(X)(t_2) \big\|_\theta^p\bigg]=0.
	\end{align}
	Similar arguments can handle the deterministic convolution $S*F(X)$:
	\begin{align}\label{con-mean3}
	\lim_{t_1\rightarrow t_2}
	\ee\bigg[\big\|S*F(X)(t_1)-S*F(X)(t_2) \big\|_\theta^p\bigg]=0.
	\end{align}
	Combining the estimations \eqref{con-mean1}--\eqref{con-mean3}, we derive
	\eqref{mean}.
\end{proof}

\begin{rk}
When $\beta=\theta=0$, the result of Theorem \ref{main1} holds if $S(\cdot)$ is only a $C_0$-semigroup.
\end{rk}

\section{Optimal Trajectory Regularity}
\label{sec4}

In this section, we consider the trajectory regularity for the solution of Eq. \eqref{spde}.
For convenience, we use the notation $\mathbb L^p(\Omega;\CC^\delta([0,T];E^\theta))$ with $\delta\in  [0,1]$ and $\theta\ge 0$ to denote $E^\theta$-valued stochastic processes $\{X(t):\ t\in [0,T]\}$ such that for any $t_1,t_2\in [0,T]$ with $t_1\neq t_2$,
\begin{align*}
&\lim_{t_1\rightarrow t_2} \|X(t_1)-X(t_2)\|_\theta=0
\quad \text{a.s. and}\quad
\ee\Big[\sup_{t\in [0,T]} \|X(t)\|_\theta^p\Big]<\infty
\end{align*}
when $\delta=0$ and 
\begin{align*}
\|X(t_1)-X(t_2)\|_\theta\le \Theta |t_1-t_2|^\delta
\quad \text{a.s. and}\quad
\ee\big[\Theta ^p\big]<\infty
\end{align*}
when $\delta>0$.
Our aim is to find the optimal constants $\delta$ and $\theta$ such that the solution of Eq. \eqref{spde} is in  
$\mathbb L^p(\Omega;\CC^\delta([0,T];E^\theta))$. 

To obtain the trajectory regularity for the solution of Eq. \eqref{spde}, in addition to Assumption \ref{a1} we propose the following assumption.

\begin{ap} \label{a2}
	There exists a constant $\alpha\in (1/p,1/2)$ with $p>2$ such that
	\begin{align*}
	K^{\theta,\alpha}_{F}:=\int_0^T t^{-\alpha} K_{F,\theta} (t) {\rm d}t<\infty,\quad 
	K^{\theta,\alpha}_{G}:=\int_0^T t^{-2\alpha} K_{G,\theta}^2 (t) {\rm d}t<\infty.
	\end{align*}
\end{ap}

Our main idea is to apply the factorization method established in \cite{DKZ87(STO)}, which associates a linear operator $R_\alpha$ with $\alpha\in (0,1)$ defined by 
\begin{align}\label{ra}
R_\alpha f(t):=\int_0^t (t-r)^{\alpha-1} S(t-r) f(r) {\rm d} r,\quad t\in [0,T],
\end{align} 
and the following generalized characterization in \cite{HL17}.

\begin{lm}  \label{prop-hol}
Let $1/p<\alpha<1$ and $\theta,\theta_1,\delta\ge 0$.
	Then $R_\alpha$ defined by \eqref{ra} 
	is a bounded linear operator from $\mathbb L^p(0,T; E^\theta)$ to $\CC^{\delta}([0,T]; E^{\theta_1})$ when $\alpha,\theta,\theta_1,\delta$ satisfy one of the following conditions:
	\begin{enumerate}
		\item
		$\delta=\alpha-\frac1p+\frac{\theta-\theta_1}2$ when $\theta_1>\theta$ and 
		$\alpha>\frac{\theta_1-\theta}2+\frac1p$;
		
		\item
		$\delta<\alpha-\frac1p$ when $\theta_1=\theta$;
		
		\item
		$\delta=\alpha-\frac1p$ when $\theta_1<\theta$.
	\end{enumerate}
\end{lm}

\begin{proof}
It is a straightforward consequence of \cite[Proposition 4.1]{HL17} and we omit the details.
\end{proof}

By stochastic Fubini theorem, the following factorization formula is valid for the stochastic convolution $S\diamond G(X)$: 
\begin{align*}
S\diamond G(X)(t)
=\frac{\sin(\pi \alpha)}{\pi}R_\alpha G_\alpha(t),
\end{align*}
where $G_\alpha(t):=\int_0^t (t-r)^{-\alpha} S(t-r) G(X(r)){\rm d}W_U(r)$, $t\in [0,T]$.

\begin{tm}  \label{main2}
In addition to the conditions of Theorem \ref{main1}, let Assumption \ref{a2} hold.
Then the following statements hold.
	\begin{enumerate}
		\item
		When $\theta=0$, 
\begin{align}\label{con-0}
X \in~ 
& \mathbb L^p(\Omega;\CC^\delta([0,T]; E)) 
\cup \mathbb L^p(\Omega;\CC^{\alpha-\frac1p-\frac{\theta_1} 2}([0,T]; E^{\theta_1})) \nonumber \\
& \cap \mathbb L^p(\Omega;\CC^{\frac{\beta-\theta_2}2\wedge 1}([0,T]; E^{\theta_2}))
\end{align} 
for any $\delta\in [0,\alpha-1/p)$, $\theta_1\in (0,2\alpha-2/p)$ and $\theta_2\in [0,\beta]$.

		\item
		When $\theta>0$, 
\begin{align}\label{con-gamma}
X \in~
&\mathbb L^p(\Omega;\CC^\delta ([0,T]; E^\theta)) 
\cup \mathbb L^p(\Omega;\CC^{\alpha-\frac1p} ([0,T]; E^{\theta_1})) \nonumber \\
&\cup \mathbb L^p(\Omega;\CC^{\alpha-\frac1p+\frac{\theta-\theta_2}2} ([0,T]; E^{\theta_2}))
\cap \mathbb L^p(\Omega;\CC^{\frac{\beta-\theta_3}2\wedge 1}([0,T]; E^{\theta_3}))
\end{align} 
for any $\delta\in [0,\alpha-1/p)$, $\theta_1\in (0,\theta)$, 
$\theta_2\in (\theta, \theta+2\alpha-2/p)$ and
$\theta_3\in [0,\beta]$.
\end{enumerate}
\end{tm}

\begin{proof}
For the initial datum, by \eqref{ana} we get
	\begin{align*}
	\|S(t_2)X_0-S(t_1)X_0\|_{\theta_1} 
	&=\|(-A)^\frac{\theta_1-\beta}2 (S(t_2-t_1)-{\rm Id}_{E}) 
	(-A)^\frac\beta2 S(t_1) X_0\| \nonumber \\
	&\le C|t_2-t_1|^{\frac{\beta-\theta_1}2\wedge 1} \|X_0\|_\beta
	\end{align*}
	for any $\theta_1\in [0,\beta)$. 
	Combining with the strong continuity of $S(\cdot)$ shows that 
	$S(\cdot) X_0\in \CC^{\frac{\beta-\theta_1}2\wedge 1}([0,T]; E^{\theta_1})$ for any $\theta_1\in [0,\beta]$. 
	
By Fubini theorem and the usual Burkholder inequality, we get
\begin{align*}
& \ee \bigg[\|G_\alpha(t)\|_{\mathbb L^p(0,T;E^\theta)}^p  \bigg]  
=\int_0^T \ee\bigg[\|G_\alpha(t) \|_{\theta}^p \bigg]  {\rm d}t  \\
&\le \bigg[\int_0^T \bigg(\int_0^t r^{-2\alpha} K_{G,\theta}^2(r) {\rm d}r  \bigg)^\frac p2 {\rm d}t \bigg] \bigg(1+\|X\|_{\HHH^p_\theta}\bigg)^p
<\infty,
\end{align*}
which implies that $G_\alpha\in \mathbb L^p(0,T; E^{\theta})$ almost surely. 
Applying Lemma  \ref{prop-hol} shows that if $\theta=0$, then
\begin{align*}
S\diamond G(X)  \in~ 
& \mathbb L^p(\Omega;\CC^\delta([0,T]; E) 
\cup \mathbb L^p(\Omega;\CC^{\alpha-\frac1p-\frac{\theta_1} 2}([0,T]; E^{\theta_1})) 
\end{align*} 
for any $\delta\in [0,\alpha-\frac1p)$ and $\theta_1\in (0,2\alpha-\frac2p)$; and if $\theta>0$, then
\begin{align*}
S\diamond G(X)  \in~
&\mathbb L^p(\Omega;\CC^\delta ([0,T]; E^\theta) 
\cup \mathbb L^p(\Omega;\CC^{\alpha-\frac1p} ([0,T]; E^{\theta_1}))  \\
&\cup \mathbb L^p(\Omega;\CC^{\alpha-\frac1p+\frac{\theta-\theta_1}2} ([0,T]; E^{\theta_2}))
\end{align*} 
for any $\delta\in [0,\alpha-1/p)$, $\theta_1\in (0,\theta)$, 
and $\theta_2\in (\theta, \theta+2\alpha-2/p)$.
Similar argument yields the same regularity for $S*F(X)$.
Thus we conclude \eqref{con-0} and \eqref{con-gamma} by combining the H\"older continuity of $S(\cdot)X_0$, $S*F(X)$ and $S\diamond G(X)$.
\end{proof}

\begin{rk}
Let $\beta=\theta=0$ and Assumptions \ref{a1} and \ref{a2} hold, then by \cite[Proposition 5.9]{DZ14}, $X\in \mathbb L^p(\Omega;\CC([0,T]; \hh))$ if $S(\cdot)$ is only a $C_0$-semigroup $S(\cdot)$.
Moreover, similarly to \cite[Theorem 2.2]{HL17}, the following stronger moments' estimation holds for some constant $C=C(T,p,\alpha,K_F^\alpha,$ $K_G^\alpha)$:
	\begin{align*}
	\ee\Big[\sup_{t\in [0,T]}\|X(t)\|^p\Big]
	&\le C\Big(1+\ee\Big[\|X_0\|^p \Big]\Big).
	\end{align*}
\end{rk}

\section{Example}
\label{sec5}

In this section, we illustrate our results by the stochastic heat equation
\begin{align}\label{she}\tag{SHE}
\begin{split}
&{\rm d} X(t,\xi)
=(\Delta X(t,\xi)+f(X(t,\xi))) {\rm d}t +g(X(t,\xi)) {\rm d}W(t,\xi),\\
&X(t,\xi)=0, \quad (t,\xi)\in [0,T]\times \partial \OOO,\\
&X(0,\xi)=X_0(\xi), \quad \xi\in \OOO,
\end{split}
\end{align}
and give a positive answer to the open problem given in \cite{JR12(JDE)}.
Without loss of generality, we assume that $X_0$ is deterministic function.

Set $U=\mathbb L^2(\OOO)$, $E=\mathbb L^q(\OOO)$ with $q\ge 2$ throughout this section.
Define $A=\Delta$ with domain $\text{Dom}(A)=W_0^{1,q}(\OOO)\cap W^{2,q}(\OOO)$.
Here $W^{\bs,q}(\OOO)$ with $\bs\in \nn_+$ denotes the Sobolev space consisting of functions in $\OOO$ whose derivatives up to order 
$\bs$ are all in $\mathbb L^q(\OOO)$, and $W_0^{1,q}(\OOO):=\{f\in W^{1,q}(\OOO):\ f|_{\partial \OOO}=0\}$.
In the following, we also need the Sobolev--Slobodeckij space $W^{\theta,q}$ with $\theta\in (0,1)$, whose norm is defined by
\begin{align*}
\|X\|_{W^{\theta,q}(\OOO)}
:=\bigg(\|X\|_{\mathbb L^q(\OOO)}^q 
+\int_{\OOO}\int_{\OOO} 
\frac{|X(\xi)-X(\eta)|^q}{|\xi-\eta|^{{\rm d}+\theta q}}
{\rm d}\xi {\rm d}\eta \bigg)^\frac1q.
\end{align*}
Similarly, we denote $W_0^{\theta,q}(\OOO):=\{f\in W^{\theta,q}(\OOO):\ f|_{\partial \OOO}=0\}$.

Let ${\bf Q}\in \LL(\mathbb L^2(\OOO))$ and $\{h_n\}_{n\in \nn_+}\subset \mathbb L^2(\OOO)$ be eigenfunctions of ${\bf Q}$ which forms an orthonormal basis of $\mathbb L^2(\OOO)$ with related eigenvalues $\{\lambda_n\}_{n\in \nn_+}$. 
Assume that $W=\{W(t):\ t\in [0,T]\}$ is an $\mathbb L^2(\OOO)$-valued ${\bf Q}$-Wiener process, i.e.,
\begin{align*}
W(t)=\sum_{n\in \nn_+}\sqrt{\lambda_n} h_n \beta_n(t),
\quad t\in [0,T],
\end{align*} 
where $(\beta_n)_{n\ge 1}$ is a sequence of independent standard Brownian motions with respect to the filtration 
$(\FFF_t)_{t\in [0,T]}$.
Our main assumption on the system $\{(\lambda_n,h_n)\}_{n\in \nn_+}$ is that there exists a constant $\epsilon \in [0,1]$ such that 
\begin{align}\label{con-q}
\sum_{n\in \nn_+} \sqrt \lambda_n \|h_n\|_{\CC^\epsilon(\OOO)}
=: C_{\bf Q}<\infty,
\end{align}
where $\|\cdot\|_{\CC^0(\OOO)}$ denotes the $\mathbb L^\infty(\OOO)$-norm and $\|\cdot\|_{\CC^\epsilon(\OOO)}$ with $\epsilon\in (0,1]$ denotes the H\"older norm
\begin{align*}
\|h\|_{\CC^\epsilon(\OOO)}:=\|h\|_{\mathbb L^\infty(\OOO)}+\sup_{\xi,\eta\in\OOO,\xi\neq \eta} 
\frac{|h(\xi)-h(\eta)|}{|\xi-\eta|^\epsilon},
\quad h\in \CC^\epsilon(\OOO).
\end{align*}

Assume that $f,g:\rr\rightarrow \rr$ are Lipschitz continuous functions with Lipschitz constant $L_f,L_g>0$, i.e., for any $\xi_1$, $\xi_2\in \rr$,
\begin{align*}
|f(\xi_1)-f(\xi_2)|\le L_f |\xi_1-\xi_2|,
\quad |g(\xi_1)-g(\xi_2)|\le L_g |\xi_1-\xi_2|.
\end{align*}
Define the operators $F:E\rightarrow E$ and $G:E\rightarrow \LL(U,E)$ by the Nymiskii operator associated with $f$ and $g$, respectively: 
\begin{align*}
F(u)(\xi):=f(u(\xi)),\quad
G(u)h_n(\xi):=\sqrt \lambda_n g(u(\xi)) h_n(\xi),
\quad \xi\in \OOO,\ n\in \nn_+.
\end{align*}
Then the stochastic heat equation \eqref{she} is equivalent to Eq. \eqref{spde}. 

Let $\theta\in (0,1)$ and fix $x,y\in \mathbb L^q (\OOO)$ and $z\in E^\theta(\OOO)$.
For the drift term, by the uniform boundedness of the semigroup $S(\cdot)$ and the Lipschitz continuity of $f$, we get the conditions on $F$ of Assumption \ref{a1}:
\begin{align*}
\|& S(t) (F(x)-F(y) \|_{\mathbb L^q(\OOO)}
\le C\|x-y\|_{\mathbb L^q(\OOO)}.
\end{align*}
On the other hand,
\begin{align*}
 \|S(t) F(z) \|_{E^\theta(\OOO)}
&\le C t^{-\frac\theta2} (1+\|z\|_{\mathbb L^q(\OOO)})
\le Ct^{-\frac\theta2}  (1+\|z\|_{E^\theta(\OOO)}).
\end{align*}
These two inequalities show the conditions on $F$ of Assumptions \ref{a1} and \ref{a2} with $K_F=C$ and 
$K_{F,\theta}=C t^{-\frac\theta2}$
for any $\alpha\in (0,1-\theta/2)$, where $C$ depends on $M(T)$, $L_f$, $f(0)$ and the volume of $\OOO$.

In view of the diffusion term, the uniform boundedness of the semigroup $S(\cdot)$, the definitions of the $\gamma$-radonifying operator and the Lipschitz continuity of $g$ and \eqref{con-q} lead to
\begin{align*}
& \|S(t) (G(x)-G(y))\|_{\gamma(\mathbb L^2(\OOO),\mathbb L^q(\OOO))} \\
& \le \sum_{n\in \nn_+} \|S(t) (G(x)-G(y)) h_n\|_{\mathbb L^q(\OOO)} \\
& \le C \Big(\sum_{n\in \nn_+} \sqrt \lambda_n \|h_n\|_{\mathbb L^\infty(\OOO)} \Big) \big\|x-y\big\|_{\mathbb L^q(\OOO)}  \\
&\le C \big\|x-y\big\|_{\mathbb L^q(\OOO)}.
\end{align*}

To verify the growth condition of $G$ in $E^\theta$ for 
$\theta>0$, we need the fact (see, e.g., \cite{Lun09}) that 
\begin{align}\label{emb}
W_0^{\theta_1,q}
\hookrightarrow E_q^{\theta_2}  
\hookrightarrow  W_0^{\theta_3,q}
\quad \text{for all}\quad 1/q<\theta_3<\theta_2<\theta_1<1.
\end{align}
Assume that $g(0)=0$ or $h_n|_{\partial \OOO}=0$.
Now let $\sigma\in (0,\theta/2)$ be sufficiently small, then \eqref{ana} yield that  
\begin{align*}
\|S(t) G(z)\|_{\gamma(\mathbb L^2(\OOO),E^\theta(\OOO))}
& \le \sum_{n\in \nn_+} \|S(t) G(z)h_n \|_{E^\theta(\OOO)} \\
&\le C t^{-\sigma} \sum_{n\in \nn_+} \|G(z)h_n \|_{E_q^{\theta-2\sigma}(\OOO)} \\
&\le C t^{-\sigma} \sum_{n\in \nn_+} \|G(z)h_n \|_{W^{\theta-\sigma,q}(\OOO)}.
\end{align*}

By triangle inequality and the boundedness of $\{h_n\}_{n\in \nn_+}$,
\begin{align*}
&\|G(x)h_n\|_{W^{\theta,q}(\OOO)}^q\\
&\le \lambda_n^\frac q2 
\bigg(\|h_n\|_{\mathbb L^\infty}^q \|g(x)\|_{\mathbb L^q(\OOO)}^q
+\int_{\OOO\times\OOO}\frac{|g(x(\xi))-g(x(\eta))|^q|h_n(\xi)|^q}{|\xi-\eta|^{{\rm d}+\theta q}}{\rm d}\xi{\rm d}\eta\\
&\qquad \qquad +\int_{\OOO\times\OOO}
\frac{|g(x(\eta))|^q |h_n(\xi)-h_n(\eta)|^q}{|\xi-\eta|^{{\rm d}+\theta q}}{\rm d}\xi{\rm d}\eta\bigg) \\
&\le C \lambda_n^\frac q2 \bigg( \|h_n\|_{\mathbb L^\infty}^q 
 \|g(x)\|_{W^{\theta,q}(\OOO)}^q
 +\bigg(\sup_{\xi,\eta\in\OOO,\xi\neq \eta} 
\frac{|h_n(\xi)-h_n(\eta)|}{|\xi-\eta|^\epsilon}\bigg)^q \\
&\qquad \qquad \times \int_{\OOO\times\OOO}|g(x(\xi))|^q|\xi-\eta|^{(\epsilon-\theta) q-{\rm d}}{\rm d}\xi{\rm d}\eta\bigg).
\end{align*}
For any $\theta<\epsilon$, we obtain by applying Fubini theorem 
\begin{align*}
& \int_{\OOO\times\OOO} |g(x(\xi))|^q|\xi-\eta|^{(\epsilon-\theta) q-{\rm d}}{\rm d}\xi{\rm d}\eta \\
&\le \bigg(\int_{\OOO-\OOO}|x|^{(\epsilon-\theta) q-{\rm d}}{\rm d}\xi\bigg) \bigg(\int_{\OOO}|g(x(\xi))|^q{\rm d}\xi\bigg)
\le C \|g(x)\|_{\mathbb L^q(\OOO)}^q,
\end{align*}
where $\OOO-\OOO$ denote the set $\{\zeta=\xi-\eta:\ \xi,\eta\in \OOO\}$.
Then we get
\begin{align*}
\sum_{n\in \nn_+} \|G(x)h_n\|_{W^{\theta,q}(\OOO)}
\le C (1+ \|x\|_{W^{\theta,q}(\OOO)}).
\end{align*}
As a result, when $\theta-\sigma<\epsilon$, we have
\begin{align*}
\|S(t) G(z)\|_{\gamma(\mathbb L^2(\OOO),E^\theta(\OOO))}
&\le C t^{-\sigma} (1+\|z\|_{W^{\theta-\sigma,q}(\OOO)}) \\
&\le C t^{-\sigma} (1+\|z\|_{E^\theta (\OOO)}).
\end{align*}
The above two inequalities yield the conditions on $G$ of Assumptions \ref{a1} and \ref{a2} with $K_G=C$ and 
$K_{G,\theta}=C t^{-\sigma}$
for any $\alpha\in (0,1/2-\sigma)$, where $C$ depends on $M(T)$, $L_f$, $f(0)$, the volume of $\OOO$ and the embedding constants in \eqref{emb}.

Thus we have shown Assumptions \ref{a1} and \ref{a2} for any 
$\theta\in [0, \epsilon]\cap [0,1/2)$ and $\alpha\in [0,1/2)$. 
Applying Theorems \ref{main1} and \ref{main2} and appropriate Sobolev embedding, we deduce the following regularity for the solution of the stochastic heat equation \eqref{she}, which gives an answer to the open question in \cite[Section 4]{JR12(JDE)}.

\begin{tm} \label{main3}
Assume that $f,g:\rr\rightarrow \rr$ are Lipschitz continuous functions, and $W$ is ${\bf Q}$-Wiener process such that the eigensystem 
$\{(\lambda_n,h_n)\}$ of ${\bf Q}$ satisfies \eqref{con-q} for some $\epsilon\in [0,1]$.
\begin{enumerate}
\item
When $\epsilon=0$, assume that $X_0\in E_q^1(\OOO)$ for any $q\ge 2$, then Eq. \eqref{she} possesses a unique mild solution $X$ such that
\begin{align}\label{she-hol0}
X \in~ \CC^\delta([0,T]; \CC^\theta(\OOO))
\quad {\rm a.s.} 
\end{align}
for any $\delta,\theta\ge 0$ such that $\delta+\theta/2<1/2$.

\item
When $\epsilon\in (0,1]$, assume that $X_0\in E_q^{3/2}(\OOO)$ for any $q\ge 2$ and that $g(0)=0$ or $h_n=0$ on $\partial \OOO$.
Then Eq. \eqref{she} possesses a unique mild solution $X$ such that
\begin{align}\label{she-hol}
X\in~ & \CC^\delta([0,T]; \CC^\kappa(\OOO)) \cup  \CC^{\delta_1}([0,T]; \CC^{\kappa_1}(\OOO))
\quad {\rm a.s.} 
\end{align} 
for any $\delta\in [0,1/2)$, $\kappa\in [0,\epsilon \wedge 1/2)$ and any $\delta_1\in [0,(1+\epsilon-\kappa_1)/2\wedge (3-2\kappa_1)/4)$, 
$\kappa_1\in [\epsilon \wedge 1/2,(1+\epsilon)\wedge 3/2)$.
\end{enumerate}
\end{tm}

\begin{proof}
When $\epsilon=0$, by Theorems \ref{main1} and \ref{main2}, Eq. \eqref{she} possesses a unique mild solution such that $X \in~ \CC^\delta([0,T]; E^\theta)$ a.s. for any $\delta,\theta\ge 0$ such that $\delta+\theta/2<1/2$.
The first H\"older continuity \eqref{she-hol0} then follows from the Sobolev embedding.

When $\epsilon\in (0,1]$, we only prove the case $\epsilon\in [1/2,1]$, since similar arguments yield the case $\epsilon\in (0,1/2)$.
Theorems \ref{main1} and \ref{main2} imply that Eq. \eqref{she} possesses a unique mild solution such that 
\begin{align*}
X  \in \CC^\delta ([0,T]; E^\theta(\OOO)) \cup \CC^{\delta_1} ([0,T]; E_q^{\theta_1}(\OOO))\quad {\rm a.s.} 
\end{align*} 
for any $\delta\in [0,1/2)$, $\theta\in [0,1/2)$, 
$\theta_1\in [1/2, 3/2)$ and $\delta_1<3/4-\theta_1/2$.
Then using the Sobolev embedding, we get
\begin{align*}
X\in~ & \CC^\delta([0,T]; \CC^\kappa(\OOO)) \cup 
\CC^{\delta_1}([0,T]; \CC^{\kappa_1}(\OOO))
\quad {\rm a.s.} 
\end{align*} 
for any $\delta\in [0,1/2)$, $\kappa\in [0,\theta-{\rm d}/q)$, 
$\delta_1\in [0,3/4-\theta_1/2)$ and 
$\kappa_1\in (0,\theta_1-{\rm d}/q)$ when ${\rm d}<q/2$.
Taking sufficiently large $q$, we conclude \eqref{she-hol}.
\end{proof}

\begin{rk}
The H\"older continuity \eqref{she-hol} shows that the solution of the stochastic heat equation \eqref{she} enjoys values in the space $\CC([0,T];\CC^\kappa(\OOO))$ of continuous differentiable functions on $[0,T]\times \OOO$ with $(\kappa-1)$-H\"older continuous spatial derivatives for any $\kappa\in (1,(1+\epsilon)\wedge 3/2)$.
\end{rk}

\bibliographystyle{amsalpha}
\bibliography{bib}

\end{document}